\documentclass[12pt,twoside,letter]{amsart}
\usepackage{mathrsfs}
\usepackage{mathrsfs}
\usepackage{txfonts}
\usepackage{bbm}
\usepackage{mathrsfs}
\usepackage{amsfonts}
\usepackage{amsmath}
\usepackage{amssymb}
\usepackage{wasysym}
\usepackage{psfrag}
\usepackage{graphics,epsfig,amsmath}  
\usepackage{comment}
\usepackage{xcolor}
\usepackage{enumerate}
\newtheorem{theorem}{Theorem}

\newtheorem{Remark}{Remark}
\newtheorem{lemma}{Lemma}
\newtheorem{prop}{Proposition}

\newtheorem{alg}{Algorithm}

\def\torus{{\mathbb T}}
\def\real{{\mathbb R}}

\def\dist{{\rm dist}}

\title[Resonant quasi-periodic equilibria]{Resonant equilibrium configurations in quasi-periodic media: {KAM} theory}

\author[R. de la Llave]{Rafael de la Llave}
\address{School of Mathematics\\
Georgia Institute of Technology \\
686 Cherry St. \\
Atlanta GA 30332, USA}
\address{
JLU-GT joint institute for Theoretical Science\\
Jilin University\\
Changchun, 130012, CHINA}
\email{rafael.delallave@math.gatech.edu}

\author[X. Su] {Xifeng Su}
\address{School of Mathematical Sciences\\
Beijing Normal University\\
No. 19, XinJieKouWai St.,HaiDian District\\
 Beijing 100875, P. R. China}
\email{xfsu@bnu.edu.cn}

\author[L. Zhang]{Lei Zhang}
\address{School of Mathematics\\
Georgia Institute of Technology \\
686 Cherry St. \\
Atlanta GA 30332, USA}
\email{lzhang98@math.gatech.edu}

\thanks{L. Z and R. L. supported by NSF grant DMS-1500943} 
\thanks{X. S supported by National Natural Science Foundation of China (Grant No. 11301513) and ``the Fundamental Research Funds for the Central Universities"} 

\begin{document}
\maketitle 
\vspace{0.1in}

\begin{abstract}
We develop an a-posteriori KAM theory for the equilibrium equations for 
quasi-periodic solutions in a quasi-periodic Frenkel-Kontorova 
model when the frequency of the solutions resonates with the 
frequencies of the substratum.

The KAM theory we develop is very different both in the methods
and in the conclusions from the more customary KAM theory for
Hamiltonian systems or from the KAM theory in quasi-periodic media 
for solutions with frequencies which are Diophantine with respect to the frequencies of the media. The main difficulty is that we cannot use transformations (as in the Hamiltonian case) nor Ward identities
 (as in the case of frequencies Diophantine with those of the media).

The technique we use is to add an extra equation that ensures the linearization of the equilibrium equation factorizes. To solve the extra equation requires
an extra counterterm. We compare this phenomenon with other phenomena in KAM theory. It seems that this technique could be used in several other problems.

As a conclusion, we obtain that the perturbation expansions developed in the previous 
paper \cite{SuZL15} converge when the potentials are in a codimension one manifold in a space of potentials. 
The method of proof also leads to efficient (low storage requirements and 
low operation count) algorithms to compute the quasi-periodic solutions.

\end{abstract} 



\section{Introduction}
The goal of this paper is to develop a KAM theory for 
the functional equation:
\begin{equation}\label{eq:main}
v(\psi + \Omega) +
v(\psi - \Omega)
- 2 v(\psi)  + W((\psi, \eta)  + \beta v(\psi)) +\lambda = 0
\end{equation}
where $W:\torus^d \rightarrow \real$, $\Omega \in \real^{d-1}$, $\beta \in \real^{d}$ 
are given, $\psi$ is a variable in $\torus^{d-1}$ and we can 
think of $\eta \in \torus$ as a parameter. We are to find $v:\torus^{d-1}\times\torus^{1}\rightarrow \real$ as a function of $(\psi,\eta)$ 
and $\lambda \in \real$ as a function of the parameter $\eta$. We will refer to \eqref{eq:main} as the ``equilibrium equation''.

The equation \eqref{eq:main} was derived in 
the paper \cite{SuZL15} as 
the equation is satisfied by hull functions of quasi-periodic 
equilibria in Frenkel-Kontorova models of deposition on
quasi-crystals when the frequency of the equilibrium solution is 
resonant with the frequencies of the substratum. The variable $\eta$ is an angle which has the meaning of a transversal phase.

Roughly, the model describes particles interacting with their
neighbors and with a substratum which is quasi-periodic. The configuration describing the state of the system is parameterized by the hull function $v$. We try to place 
the particles with a frequency (inverse of the density of particles) which resonates with the frequencies of the medium. 
$\Omega$ represents the \emph{intrinsic frequencies}. Since the medium is 
resonating with the frequency of the configuration, the positions of the particles are parameterized by $d-1$ angles, i.e. they cover densely a $d-1$ dimensional torus which is indexed by $\eta$.
The $W$ represents the forces of the particles with the substratum and the 
$\lambda$  is an external force. 

We refer to the paper \cite{SuZL15} for the discussion of the physical 
motivations (there are several physical motivations for the Frenkel-Kontorova model)
and for a formal analysis. From the strictly logical point of 
view, this paper and \cite{SuZL15} are completely independent even if they are motivated by the same physical problem. They also rely on very different techniques. 
To avoid repetition, we refer to \cite{SuZL15} for 
references to the previous literature on the problem as well as for physical motivations. 

The main goal of this paper is to develop a KAM theory for the equation 
\eqref{eq:main}, but we will have to add 
a one dimensional extra parameter to it. 

The main source of the difficulty to implement a
Newton method -- as needed in KAM theory --  is that the equation \eqref{eq:inv_lin} (the linearization of 
the equilibrium equation \eqref{eq:main}) is 
not easily analyzable in a way that leads to tame estimates. We will deal with this problem by adding an extra auxiliary equation which implies that \eqref{eq:inv_lin} can be solved with tame estimates. The addition of an extra equation that allows to solve the linearization is similar in spirit to the introduction of the reducibility in KAM theory \cite{Moser'67}.
Nevertheless, our auxiliary equation is very different from the one in reducibility.

\subsection{The method of adding extra parameters to equations} 

The main observation that allows us 
to develop a KAM theory is that if we are allowed to adjust a 
one dimensional parameter in the potential, then the linearized equilibrium equation
admits a very nice structure (it can be factorized into two 
first order equations and the factorization allows to develop an iterative procedure which is quadratic convergent).

Hence, in Section \ref{sec:modified equilibrium equation}, we will add an extra parameter to the left hand side of \eqref{eq:main} to obtain the modified equilibrium equation \eqref{external force} such that its linearization can be factorized. Then, we  supplement \eqref{external force} with another equation \eqref{modified factorization} in Section \ref{sec:factorization equation} (we call it \emph{the factorization equation}) 
which encodes that the linearization of the equilibrium equation can be solved.

\subsubsection{The modified equilibrium equation}\label{sec:modified equilibrium equation}
For each fixed  $\eta \in \mathbb{T}^1$
we will look for a function
$v:\mathbb{T}^{d-1}\rightarrow \mathbb{R}$ and for numbers
$\lambda, \sigma$ in such a way that we have
\begin{equation}\label{external force}
\begin{split}
&\mathscr{E}[v, \sigma,\lambda](\psi, \eta)\\
\equiv &
v(\psi +\Omega) + v(\psi -\Omega) -
2v(\psi) + W((\psi, \eta) + \beta v(\psi)) + v(\psi) \sigma+
\lambda\\
= & 0.
\end{split}
\end{equation}
In the rest of the paper, we will call this modified equation \eqref{external force} the equilibrium equation. The equation \eqref{external force} has a symmetry that makes the solutions not unique (this corresponds to a gauge symmetry related to the choice of origin of the phase in the original problem). Hence, to obtain local uniqueness, we supplement \eqref{external force} with the following normalization: 

\begin{equation}\label{normalization}
\int_{\mathbb{T}^{d-1}} v(\psi) \, d \psi  = 0.
\end{equation}

We consider $\mathscr{E}$ as a functional that given a function $v$ and two numbers $\sigma, \lambda$ produces another function given by the second line in \eqref{external force}. And we treat the equation \eqref{external force} as searching for zeros of the functional $\mathscr{E}$. 

Since $v,\sigma,\lambda$ all depend on the parameter $\eta$, for convenience, we will write $v_{\eta}$ when we need to emphasize the fact that $v$ depends on $\eta$, and similarly for all the other functions. We will also obtain $\sigma$, $\lambda$ as a function of $\eta$ (and maybe of frequency $\Omega$, but we will not discuss dependence on $\Omega$ in this paper).

Depending on the physical solutions, we may impose the value of one of these variables and determine the others. For example, if we are imposing an external force we may want to fix $\lambda$ or if the material is constrained to have a certain density, we may fix $\Omega$.

It is important to notice that, once we have established the KAM theorem, eliminating some variables in terms of the others is just an application of the finite dimensional implicit function theorem.

\begin{Remark}
Variants of the idea of adding external parameters and then setting them
to zero, has appeared in many guises. In perturbative expansions in
Physics, it is called \emph{the method of counterterms}
\cite{BogoliubovS80, Gallavotti85}. In differential equations, it is
called \emph{Cesari's alternative method} (Chapter IX of \cite{Hale80},\cite{ChowH82}).
Closer to us, in KAM theory, it was introduced in \cite{Moser'67}.
It was realized in \cite{Llave86,Broer96,Yoccoz92, Sevryuk99} that it provided a
good way to deal with degenerate problems. A very systematic
treatment of dependence on parameters (including
parameters taking values in nowhere dense sets) appears in \cite{Vano02}.
\end{Remark}

\begin{Remark} 
The exact form of the counterterm added is not that important. 
We could have put other counterterms 
$\sigma F(v)$ for almost any function $F\nequiv0$. 

The way of thinking geometrically of the counterterms is 
that there is  codimension $1$ set of potentials for 
which the solutions move differentiably with respect to 
parameters. The counterterm is a projection that moves to keep the problem in this manifold. We could have 
taken any other family of corrections to the codimension one 
manifold where the solutions are found. We refer to Section \ref{sec:extensions} for a precise formulation and more details.
\end{Remark} 

\begin{Remark}
Even if we could consider \eqref{external force} as a functional equation for
each value of $\eta$, we will show in the next section that the
symmetries of the equation involve mixing the $\psi$ and $\eta$
dependence. Relatedly, we note that the equation
\eqref{external force} for a fixed value of $\eta$ does not have a variational
principle.

It will be important to mention that, because $\beta$ has components
both in the $\psi$ and the $\eta$ directions, the equation
\eqref{external force} cannot be considered just as a parameterized version
of the equations considered in \cite{SuL1}.
\end{Remark}

\subsubsection{The factorization equation}\label{sec:factorization equation}
For the KAM treatment of the equilibrium equation \eqref{external force}, 
we will find it useful to supplement \eqref{external force} with another equation which we call the factorization equation
\begin{equation}\label{modified factorization}
 \mathscr{F}[v,\sigma,c](\psi,\eta) \equiv [-c(\psi)+2-\partial_{\beta}W((\psi,\eta)+\beta v_\eta(\psi))-\sigma]\ c(\psi+\Omega)-1=0.
\end{equation}
This should be considered as a functional equation for functions  $v,c:\mathbb{T}^{d-1}\rightarrow \mathbb{R}$ and number $\sigma$ when $\eta$ is fixed.

The equation \eqref{modified factorization} gives a condition which ensures that the
 linearization of the equilibrium equation \eqref{external force} has a nice structure which allows to implement a Newton method with tame estimates (namely, that it factorizes into two first order difference equations). 
See Section \ref{sec:motivation} for a discussion of \eqref{modified factorization} as a tool for solving \eqref{external force}.

The  equation~\eqref{modified factorization} is hard to solve exactly, 
but we will be able to develop a quasi-Newton method for \eqref{modified factorization}.

The main idea of the paper is that, even if we do not know how to 
carry out a KAM theory for the equilibrium equation 
\eqref{external force} alone, we can carry out a KAM procedure for 
the pair of equations \eqref{external force} and \eqref{modified factorization}.
As discussed in Section~\ref{sec:reducibility}, similar things (a functional equation supplemented by another auxiliary one that makes the associated linearized equation  solvable) have happened in 
classical problems in KAM theory \cite{Moser'67,Eliasson88}. 

\begin{Remark} 
Notice that both the equilibrium equation \eqref{external force} and the factorization equations \eqref{modified factorization}
are coupled because $v$ is an unknown in both equations. 

Nevertheless, the effect of the factorization equation on the 
equilibrium equation appears only through the counterterm $\sigma$ and is a
mild effect (linear in $\sigma$). (See Section \ref{sec:solve the linearization} for a detailed treatment.) On the other hand, the effect of the variables 
of the equilibrium equation is very strong. Hence, one can 
think of the pair \eqref{external force} and \eqref{modified factorization} almost as an upper triangular system of equations. Indeed, the perturbative treatment of the pair \eqref{external force} and \eqref{modified factorization} has a skew product structure. See Section~\ref{sec:perturbative}. 
\end{Remark}

\subsubsection{Comparison of the method in this paper with the application of reducibility} 
\label{sec:reducibility}
The method of adding an extra equation so that the linearization 
of the equilibrium equation is solvable, has already appeared 
in KAM theory. 

In the theory of perturbation of lower dimensional elliptic 
tori, the classical  treatment is to try to 
reduce the linearized equation to an equation with constant coefficients
\cite{Moser'67, Melnikov}.  This requires extra non-resonance conditions 
and, in principle extra parameters. See, in particular \cite{ChengW99,HanY06,Treschev91},
which study the problem of breakdown of resonant tori in Hamiltonian systems. 

In the present case, the situation is completely different 
in the details 
(since we do not seek reducibility but rather factorization 
into two second order equations) as well as in the concepts (in \cite{Moser'67}, the parameters are related to initial conditions or the characteristic numbers of the linearized equation). Hence the parameter count of the present method is 
very different from what one could expect from reducibility. Also the reducibility equations have a very different geometric meaning from the factorization equations.


The factorization method has analogies in higher dimensional 
systems and in elliptic PDE \cite{Kozlov'83,Moser'88}.  
One can think of factorization as an analogue of putting the PDE in divergence form. The transformation to divergence form is achieved in \cite{Kozlov'83,Moser'88} for elliptic operators taking advantage of an identity (which is analogous to the Ward identities in gauge theory). Here, on the other 
hand, we have to do a KAM theorem to obtain an auxiliary function that gives the factorization. 

The present method gives an a-posteriori theorem and smooth dependence on parameters, justifies the perturbation theory and leads to efficient numerical implementations. 

\subsection{Properties of 
the equilibrium equations \eqref{external force} and its associated factorization equation \eqref{modified factorization}}

Before embarking on the hard analysis, in this section, we derive some 
identities and symmetries of the equations which are only soft 
analysis.  This section can be skipped by readers interested only in the KAM methods.

Many of the symmetries and elementary properties 
derived for the equilibrium equation in \cite{SuZL15} lift 
straightforwardly to the factorization equation. Surprisingly, the formal  perturbation theory developed here
for the pair \eqref{external force} and \eqref{modified factorization} is more efficient than the perturbation theory for \eqref{eq:main} alone developed in \cite{SuZL15}. The perturbation theory developed in \cite{SuZL15} was only for perturbation around integrable solutions, but the perturbation theory for the pair \eqref{external force} and \eqref{modified factorization} is developed around any solution of both equations \eqref{external force} and \eqref{modified factorization}. See Section \ref{sec:extensions}. 

The 
expansions around zero found in \cite{SuZL15} are a
particular case of the expansions found here since 
the linearization of the equilibrium equation around zero admits a trivial factorization.

\subsubsection{The symmetries of 
the equilibrium  equations \eqref{external force}. }

We note that the symmetries for the equilibrium equation found in 
\cite{SuZL15} extend to 
the factorization equation. 

We have that if $v_\eta, \sigma(\eta), \lambda(\eta)$ 
is a solution of \eqref{external force}, for any function $\iota(\eta)$ 
so is: 
  \begin{equation}\label{symmetry}
  \begin{split}
  \tilde v_\eta(\psi) & =
  v_{\eta + \iota(\eta)\beta_\eta}\big(\psi + \iota(\eta) \beta_\psi\big) + \iota(\eta),\\
 \tilde \sigma(\eta) &= \sigma\big(\eta + \iota(\eta)  \beta_\eta \big),\\
 \tilde \lambda(\eta) &=  \lambda \big(\eta + \iota(\eta)\beta_\eta\big) 
 - \iota(\eta) \ \sigma\big( \eta + \iota(\eta) \beta_\eta\big) .
  \end{split}
  \end{equation}

Here we use $\beta_{\psi}$ to denote the first $d-1$ components of $\beta$ and $\beta_{\eta}$ to denote the last component of $\beta$.

Since the symmetry \eqref{symmetry} involves changes of
arguments,  giving a $v_\eta$, finding the $\iota(\eta)$ that accomplishes the normalization involves solving the implicit
equation
\begin{equation}\label{implicit}
 I(\eta + \beta_\eta \iota(\eta) ) + \iota(\eta)  = 0
\end{equation}
where $I(\eta) \equiv \int_{\mathbb{T}^{d-1} } v_\eta(\psi)\, d \psi$.

Applying the finite dimensional implicit function theorem, we can
solve  \eqref{implicit}
if $I$ and its derivative with respect to $\iota(\eta)$ are both small.
In contrast, in the non-resonant case treated in \cite{SuL1}, the normalization of
the function $v$ could always be solved explicitly.

As we will prove in  Section~\ref{sec:proof}, the solutions
of \eqref{external force} that satisfy the normalization \eqref{normalization}
will be locally unique.

\section{Preliminaries}\label{spaces and linear estimates}
To formulate the KAM results (as well as to make quantitative the
Lindstedt series) we need to define precisely the norms of
analytic functions.  In this section, we collect the
definitions and some standard properties of spaces of analytic functions.

In Section \ref{spaces}, we collect several standard definitions of
spaces and present some preliminary results on these spaces. In
Section \ref{Diophantine properties} we present definitions of the
Diophantine properties we will use in this paper. In Section
\ref{cohomology} we present well known estimates for cohomology
equations, which are the basis of the KAM procedure. Besides the customary constant coefficient equations, we study first order cohomology equations with non-constant coefficients in section \ref{sec:twisted}, which were also studied in  \cite{Herman83}.

\subsection{Spaces of functions we will use}\label{spaces}
We will use a variation on the same  spaces
of analytic functions which have been used very often in
KAM theory since \cite{Moser'67}. We will use the same notations as in
\cite{Rafael'08,CallejaL'10, SuL1, SuL2}.

We denote by
\begin{equation*}
D_\rho\equiv \{~ \psi \in \mathbb{C}^{d-1}/\mathbb{Z}^{d-1} ~| \quad
|\mathrm{Im}(\psi_j)|<\rho ~\}.
\end{equation*}

We denote the Fourier expansion of a periodic function $v(\psi)$ on
$D_\rho$ by
\begin{equation*}
v(\psi)=\sum_{k\in \mathbb{Z}^{d-1}} \hat{v}_k e^{2\pi i k\cdot \psi},
\end{equation*} where $\cdot$ is the Euclidean scalar product in
$\mathbb{C}^{d-1}$ and $\hat{v}_k$ are the Fourier coefficients of $v$.

We denote by $\mathscr{A}_\rho$ the Banach space of analytic
functions on $D_\rho$ which are real for real argument and extend
continuously to $\overline{D_\rho}$. We make $\mathscr{A}_\rho$ a
Banach space by endowing it with the supremum norm:
\begin{equation*}
\|v\|_\rho =\sup_{\psi\in \overline{D_\rho}}|v(\psi)|.
\end{equation*}

The spaces of analytic functions $\mathscr{A}_\rho$ are the same
spaces as in \cite{Moser'67} and that some of their elementary
properties used in the argument were discussed in
\cite{CallejaL10,SuL1}.  Notably:
\begin{itemize}
\item
Interpolation inequalities (Hadamard three circle theorem):
\begin{equation}\label{interpolation inequalities}
\|v\|_{\theta\rho +(1-\theta) \rho'} \leq
\|v\|_{\rho}^{\theta} \ 
\|v\|_{\rho'}^{1-\theta}.
\end{equation}

\item
Cauchy inequalities:
\begin{align}
&\|D^l v\|_{\rho-\delta} \leq C(l,d)\ 
\delta^{-l}\  \|v\|_\rho, \nonumber\\
& |\hat{v}_k| \leq e^{-2\pi
\ |k|\ \rho}\  \|v\|_\rho.\nonumber
\end{align}

\item
The regularity of the composition:
\begin{prop}\label{composition}
Let $f$ be an analytic function in a domain $\mathscr{D} \subseteq \mathbb{C}, v\in \mathscr{A}_\rho$.
Assume $v(\mathbb{T}^d_\rho)\subseteq \mathscr{D}, \text{dist} (v(\mathbb{T}^d_\rho), \mathbb{C}-\mathscr{D}) \geq \xi>0$.
Then, 
\begin{itemize}
\item [(1)]  $f\circ v \in \mathscr{A}_\rho$;

\item [(2)] If $\|\tilde{v}\|_\rho < \frac{\xi}{2}$, we have
\[
\|  f(v+ \tilde{v}) - f(v) - f'(v) \tilde{v}\|_\rho \leq C \| \tilde{v}\|_\rho^2.
\]
\end{itemize}
\end{prop}
\end{itemize}

The cohomology equations we will have to consider are
different from those studied before  and we will present the results in
Section~\ref{sec:twisted}.

\begin{Remark}
The method of proof works also for spaces of functions with finite differentiability.
Indeed, KAM  theory is often formulated as an abstract implicit function
theorem for the functional $\mathscr{E}$ acting on spaces of functions
that satisfy some mild properties \cite{Schwartz60, Zehnder75}.
(The paper \cite{CallejaL10} presents an implicit function theorem
well suited for the method in this paper.)

In particular, the method in this paper works as well when
$v$ is considered in Sobolev spaces of high enough
regularity. For simplicity, we will not formulate the finite differentiable version of the results.
\end{Remark}

\subsection{Diophantine condition}\label{Diophantine properties}




We will require that $\Omega$ satisfies the Diophantine condition in $\mathbb{R}^{d-1}$:
\begin{equation}\label{Diophantine}
|k\cdot \Omega-m| \geq \kappa |\tilde{k}|^{-\tau} \qquad \forall~
k\in \mathbb{Z}^{d-1}-\{0\},~m\in \mathbb{Z}.
\end{equation}
Here $\kappa,\tau$ are positive numbers.

We denote by
$\mathscr{D}(\kappa,\tau)$ the
set of $\Omega$ which satisfy \eqref{Diophantine}.
We also denote $\mathscr{D}(\tau) = \cup_{\kappa > 0} \mathscr{D}(\kappa,\tau)$.

It is well known that the set
$\mathscr{D}(\tau)$ with $\tau > d-1$ is of
 full $d-1$ dimensional
Lebesgue measure in $\mathbb{R}^{d-1}$.


\subsection{Cohomology equations}\label{cohomology}

It is standard in KAM theory to solve for $v$ given
$\phi$  with zero average in such a way that:
\begin{equation}\label{Cohomology equation}
v(\psi+\Omega)-v(\psi)=\phi(\psi),
\end{equation}
where $\Omega\in\mathscr{D}(\kappa,\tau)$.

We will use equations similar to \eqref{Cohomology equation} frequently with some fixed frequency
$\Omega$. To simplify our notations, we will denote $v(\psi+\Omega)$
and $v(\psi-\Omega)$ as $v_+(\psi)$ and $v_-(\psi)$, respectively. Similar
notations will be used for other functions. We also use $T$ to
represent the translation operators,
i.e. $[T_{\Omega}v](\psi)=v(\psi+\Omega)$.

Estimates for \eqref{Cohomology equation} for Diophantine frequencies
were  proved in \cite{Russmann'75, Russmann'76}.
The crucial point of these estimates is that the solution is bounded
in smaller domains and that there are bounds on the solution in the smaller
domains. These estimates have large constants if the loss of domain is small, but the constants can be chosen to be a power of the domain loss (tame estimates). See \eqref{cohomology bounds for analytic fixed}.


Tame estimates (and hence Diophantine conditions) are used in  the convergence proofs and
in the KAM theory developed here.

\begin{lemma}\label{estimate lemma for cohomology equation}
Let $\phi \in\mathscr{A}_\rho(\mathbb{T}^{d-1})$  be such that
\begin{equation}\label{normalization phi}
\int_{\mathbb{T}^{d-1}}\phi (\psi)d\psi=0.
\end{equation}

Assume that $\Omega \in \mathscr{D}(\kappa,\tau)$.

Then, there
exists a unique solution $v$ of
\eqref{Cohomology equation} which satisfies
\begin{equation}\label{normalization equation}
\int_{\mathbb{T}^{d-1}} v (\psi)d\psi = 0.
\end{equation}
The solution $v$ is in $\mathscr{A}_{\rho'} $ for
any $0<\rho' < \rho$, and
\begin{equation}\label{cohomology bounds for analytic fixed}
\|v\|_{\rho^\prime}\leq C(d,\tau)\ \kappa^{-1}\
(\rho-\rho^\prime)^{-\tau}\|\phi \|_\rho.
\end{equation}

Furthermore, any distribution solution of \eqref{Cohomology
equation} differs from the solution claimed before by a constant.

If $\phi$ is such that it takes real values for real arguments,
so does $v$.

\medskip

Similarly if we consider analytic functions $\phi\in \mathscr{A}_\rho(\mathbb{T}^d)$ satisfying $\int_{\mathbb{T}^{d-1}} \phi (\psi,\eta)d\psi=0$, then for each $\eta$, we can solve
\begin{equation}\label{Cohomology equation2}
v(\psi+\Omega,\eta)-v(\psi,\eta)=\phi(\psi,\eta).
\end{equation}
The solution $v\in\mathscr{A}_{\rho}(\mathbb{T}^d)$ is analytic in $(\psi,\eta)$ and we have
\begin{equation}\label{cohomology bounds for analytic2}
\|v\|_{\rho^\prime}\leq C(d,\tau)\ \kappa^{-1}\
(\rho-\rho^\prime)^{-\tau}\|\phi\|_\rho.
\end{equation}
\end{lemma}

We note that, as it is well known that obtaining $v$
solving \eqref{Cohomology equation} for
given $\phi$  is very explicit in
terms of Fourier coefficients.
If
\[
\phi(\psi) = \sum_{k \ne 0} \hat \phi_k  e^{2 \pi i (k \cdot \psi)}
\]
then,  $v$ is given by
\[
v(\psi) = \sum_{k\ne 0} \hat \phi_k (e^{2 \pi i k \cdot \Omega} -1)^{-1} e^{2 \pi i (k \cdot \psi)}.
\]

The above formula for
the solution makes it clear that if $\phi$ is real for real values of its
arguments, so is $v$.  Note also that if the function $\phi$ is
discretized in terms of $N$ Fourier coefficients, the computation of
the Fourier coefficients of $v$ takes only $N$ operations.
It also makes it clear how to obtain the tame estimates (with a worse exponent)
for Diophantine frequencies or the results for \eqref{subexponential}.

\subsection{Cohomology equations with non-constant coefficients} 
\label{sec:twisted}
In this section we consider a generalization of the above theory for
equations with non-constant coefficients. These equations have also been called 
{\sl twisted cohomology equations} in \cite{Herman83}. 
Our strategy is similar to \cite{Herman83}, but we will need more details about the dependence of the solutions on the coefficient.

We will be considering equations of the form 
\begin{equation} 
\label{twisted} 
a(\psi) v(\psi + \Omega) - b(\psi) v(\psi) =  \lambda +  \phi(\psi) 
\end{equation} 
where $\psi\in\mathbb{T}^{d-1}$ and $a,b$ are fixed functions in $\mathscr{A}_\rho$.

We consider that $\phi\in\mathscr{A}_{\rho}$ is the known data of the equation and 
$\lambda\in\mathbb{R}$, $v\in\mathscr{A}_{\rho-\delta}$ are the unknowns we seek for. We will refer to \eqref{twisted} as ``twisted cohomology equations".

We will show that, given some non-degeneracy conditions in 
$a$ and $b$, we can obtain estimates for $\lambda$ and $v$. 
The idea, already present in \cite{Herman83}, is that we
can rewrite the equation \eqref{twisted} into a constant coefficient equation. We will go through the procedure in details because we will need rather detailed estimates on the effect of $a$ and $b$.

\begin{lemma}\label{lm:twisted}
Suppose
$||a-1||_{\rho}<r<1,~ ||b-1||_{\rho}<r<1$ and that $\Omega$ satisfies the Diophantine condition \eqref{Diophantine}. Then there exist positive real valued functions $\gamma_a, \gamma_b \in \mathscr{A}_{\rho-\delta/2}$ and 
real numbers $\bar a$, $\bar b$ such that 
\begin{equation} \label{auxiliary1}
\begin{split}
a(\psi) &= {\bar a} \frac{\gamma_a(\psi + \Omega)}{\gamma_a(\psi)}, \\ 
b(\psi) &= {\bar b} \frac{\gamma_b(\psi)}{\gamma_b(\psi + \Omega)}. \\ 
\end{split} 
\end{equation} 

In addition, for the twisted cohomology equation \eqref{twisted}, there exist a unique solution  $\lambda\in\mathbb{R}$ and $v\in\mathscr{A}_{\rho-\delta}$ for \eqref{twisted} satisfying the normalization 
\begin{equation}\label{normalization for twist cohomology}
\int_{\mathbb{T}^{d-1}}v d\psi=0
\end{equation}
such that
\begin{equation}\label{twist estimate}
  \begin{split}
    |\lambda|&\leq \frac{||\phi||_{\rho}||\gamma_a(\gamma_b)_+||_{\rho-\delta/2}}{\int_{\mathbb{T}^{d-1}} \gamma_a(\gamma_b)_+d\psi}\\
  ||\gamma_a\gamma_bv||_{\rho-\delta}&\leq C\delta^{-\tau}
||(\lambda+\phi)\gamma_a(\gamma_b)_+||_{\rho}\qquad \forall~ \delta>0.
\end{split}
\end{equation}
Therefore,
\begin{equation}\label{twist final estimate}
  ||v||_{\rho-\delta}\leq C||\gamma_a||_{\rho}\left\|1/\gamma_a \right\|_{\rho-\delta}
||\gamma_b||_{\rho}\left\|1/\gamma_b \right\|_{\rho-\delta }\delta^{-\tau}||  ( \lambda+\phi ) ||_{\rho}\qquad \forall~ \delta>0.
\end{equation}
\end{lemma}

\begin{proof}
To show \eqref{auxiliary1}, it suffices to observe that, taking logarithms, \eqref{auxiliary1} is 
equivalent to: 
\begin{equation} \label{auxiliarylog}
\begin{split}
(\log a)(\psi) &= \log( {\bar a})  +  \log(\gamma_a)(\psi + \Omega) 
- \log( \gamma_a) (\psi) \\ 
(\log b)(\psi) &= \log({\bar b}) 
+ \log( \gamma_b) (\psi) - \log(\gamma_b)(\psi + \Omega) \\ 
\end{split} 
\end{equation} 
which are cohomology equations with constant coefficients. Applying Lemma~\ref{estimate lemma for cohomology equation}, we get solutions $\gamma_a,\gamma_b\in\mathscr{A}_{\rho-\delta/2}$, $\bar{a},\bar{b}\in\mathbb{R}$ and estimates
\begin{equation} \label{estimates on log}
\begin{split}
||\log \gamma_a||_{\rho-\delta/2} &\leq C(d,\tau,\kappa)\delta^{-\tau}||\log a-\log\bar{a}||_{\rho}, \\ 
||\log \gamma_b||_{\rho-\delta/2} &\leq C(d,\tau,\kappa)\delta^{-\tau}||\log b-\log\bar{b}||_{\rho}.\\ 
\end{split} 
\end{equation} 

Once we solve the constant coefficient cohomology equation \eqref{auxiliarylog} for $\log{a}, \log{b}, \log{\gamma_a}, \log{\gamma_b}$ we obtain $a, b, \gamma_a, \gamma_b$ by taking exponentials. 
This ensures that they are positive for real values of the argument.
We refer to $\bar{a},\bar{b}$ as the average coefficients of the cohomology equation.

Once we have the solution of 
\eqref{auxiliary1}, we realize that the  
equation \eqref{twisted} is equivalent to
\begin{equation} \label{twisted modified} 
  \bar{a}(\gamma_a\gamma_b v)_+-\bar{b}(\gamma_a\gamma_b v)=(\lambda+\phi)\gamma_a(\gamma_b)_+
\end{equation}
which is a cohomology equation with constant coefficients. Let us denote $m=\gamma_a\gamma_b v$ for simplicity.

When $\bar{a}=\bar{b}$, we can solve \eqref{twisted modified} using Lemma \ref{estimate lemma for cohomology equation} and get estimates \eqref{twist estimate}. 

Note if $\bar{a}\neq \bar{b}$, the equation \eqref{twisted modified} is easier to solve since no small divisors appear. 
\begin{itemize}
\item If $\bar{a}> \bar{b}$, then \eqref{twisted modified} is equivalent to 
\[
m_+ - \frac{\bar{b}}{\bar{a}}  m = \frac{\lambda+\phi}{\bar{a}}\gamma_a(\gamma_b)_+,
\]
namely, $m_+= \frac{\lambda+\phi}{\bar{a}}\gamma_a(\gamma_b)_+ + \frac{\bar{b}}{\bar{a}}  m$.

Therefore,
\[
m(\psi) 
=\sum_{n=0}^\infty \left(\frac{\bar{b}}{\bar{a}}\right)^{n} \frac{\lambda+\phi\big(\psi - (n+1)\Omega\big)}{\bar{a}}\gamma_a\big(\psi-(n+1)\Omega\big) \gamma_b(\psi-n\Omega)
\]
is a solution of \eqref{twisted modified}. Hence, $\|m\|_\rho \leq C(\bar{a}, \bar{b}) \ 
||(\lambda+\phi)\gamma_a(\gamma_b)_+||_{\rho}$.

\item If $\bar{a} < \bar{b}$, then \eqref{twisted modified} is equivalent to 
\[
\frac{\bar{a}}{\bar{b}}m_+  -  m = \frac{\lambda+\phi}{\bar{b}}\gamma_a(\gamma_b)_+,
\]
namely, $m=  \frac{\bar{a}}{\bar{b}}  m_+ - \frac{\lambda+\phi}{\bar{b}}\gamma_a(\gamma_b)_+ $.

Therefore,
\[
m(\psi) 
=- \sum_{n=0}^\infty \left(\frac{\bar{a}}{\bar{b}}\right)^{n} \frac{\lambda+\phi (\psi + n\Omega)}{\bar{b}}\gamma_a(\psi +n\Omega) \gamma_b\big(\psi+(n+1)\Omega\big)
\]

is a solution of \eqref{twisted modified}. Hence, $\|m\|_\rho \leq C(\bar{a}, \bar{b}) \ 
||(\lambda+\phi)\gamma_a(\gamma_b)_+||_{\rho}$.
\end{itemize} 
Note that when $\bar{a}\neq \bar{b}$ we do not have any loss of domain, but the estimates depend on $\bar{a},\bar{b}$. To get estimates uniformly in $a,b$, we need to use the Fourier method.
\end{proof}

The choice of the parameter $\lambda$ so as to achieve the normalization deserves some discussion. In the case $\bar{a}=\bar{b}$ we see that we have to choose $\lambda$ in such a way that the right hand side of \eqref{twisted modified} has zero average. In this case, however, we can choose the average of the solution of \eqref{twisted modified} arbitrarily. 

In the case $\bar{a}\neq \bar{b}$, given any $\lambda$, the solution of \eqref{twisted modified} will be unique. Furthermore, the solution will be an affine function of $\lambda$ and, hence, so will be its average.

We see that the derivative of the average of the solution with respect to $\lambda$ is 
\[
\frac{d}{d\lambda} \langle m \rangle = \frac{\langle \gamma_a (\gamma_b)_+\rangle}{\bar{a} - \bar{b}}.
\]

In summary, in the equal average coefficient case, the equation requires adjusting one parameter to be solvable, but it gives back one free parameter of solutions.
In the different average coefficient case, we do not require any parameter to ensure the solutions, but the solution is unique.

In both cases, the solutions of twisted cohomology equations require as many parameters as they give back. This allows us to discuss the solutions of the equations which factorize into twisted cohomology equations and they require as many parameters as they give back.

\begin{Remark}
It will be important for subsequent applications that the estimates \eqref{twist estimate} are formulated in terms of $\gamma$. We could have attempted to formulate them in terms of $a$, but this is not practical for subsequent applications. On the other hand, we have formulated the normalization condition \eqref{normalization for twist cohomology} in a way that it is independent of $\gamma$. This will be important in our iterative scheme because $\gamma$ will change from step to step.

In subsequent applications, we will be using Lemma \ref{lm:twisted} when the $a,b$ are changing. It will be important to track the changes of $\gamma_a, \gamma_b$  when we change $a, b$. See  estimates \eqref{estimate for log_sigma_c}.

\end{Remark}

\section{The KAM theorem}
\label{KAM}
In this section, we will state precisely the main result of this paper. 

\subsection{Statement of the main result}
For the following theorem, we will fix the parameter $\eta$ and omit the subscript $\eta$.

We denote $\langle g\rangle$ the average of some function $g$ and $\tilde{g}=g-\langle g\rangle$.
We also use
the notation:
\[
\mathscr{L} = T_{\Omega}+T_{-\Omega}-2+\partial_{\beta}W((\psi,\eta)+\beta v)+\sigma.
\]
\begin{theorem}\label{analytic KAM theorem}
Let $\alpha\in\mathbb{R}^d$ such that
$\alpha\cdot j\neq0,~j\in\mathbb{Z}^d \setminus \{0\}$ and $\omega\in \mathbb{R}$ be such that $\omega \alpha$ is resonant. Let $W$ be an analytic function defined in a domain $\mathscr{D}\subset\mathbb{C}^d /\mathbb{Z}^d$. Take $\rho>0$ and $0<s<\rho/2$.

Denote:
\begin{eqnarray}
\label{translation}
\mathscr{E}(v,\sigma,\lambda)&=&v_++v_--2v+W((\psi,\eta)+\beta v)+\sigma v + \lambda,\\
\mathscr{F}(v,\sigma,c)&=&(-c+2-\partial_{\beta}W((\psi,\eta)+\beta v)-\sigma)c_+-1.
\end{eqnarray}

We assume:
\begin{itemize}
\item[(H1)] Initial guesses: Let $(v^0, c^0, \sigma^0, \lambda^0)$ be an approximate solution such that $\|\mathscr{E}[v^0,\sigma^0,\lambda^0]\|_\rho \le \epsilon$ and $\|\mathscr{F}[v^0,\sigma^0,c^0]\|_\rho \le \epsilon$.
\item [(H2)] Diophantine properties: There exists a matrix $B\in SL(d, \mathbb{Z})$,
$\Omega \in \mathbb{R}^{d-1}$, $L \in \mathbb{Z}^d$ such that
$B\omega\alpha = (\Omega, 0)  + L $ with
\begin{equation*}
|l\cdot \Omega-n| \geq \kappa |l|^{-\tau} \qquad \forall~
l\in \mathbb{Z}^{d-1}-\{0\},~n\in \mathbb{Z}.
\end{equation*}
\item [(H3)] Non-degeneracy conditions:
\begin{equation*}
\begin{split}
&||c^0-1||_{\rho}<M_1<1,\\
&|\sigma^0|<M_2<1,\\
&||W_{v^0}||_{C^2(\mathscr{D})}< M_3<1\\
&||v^0||_{\rho}<m<1,
\end{split}
\end{equation*}
where $W_{v^0}$ is short for $W((\psi,\eta) + \beta v^0)$, $M_1+M_2+M_3<1$ and $m$ depends on $d,\tau,\kappa,M_i,s$.
\item[(H4)] Composition condition: Denote
  \begin{equation*}
    \mathscr{R}_{v^0}=\{(\psi,\eta)+\beta v^0(\psi),(\psi,\eta)\in\mathbb{T}^d_{\rho}\}.
  \end{equation*}
We assume $\mathscr{R}_{v^0}\subset\mathscr{D}$ and $\dist(\mathscr{R}_{v^0},\partial\mathscr{D})\geq 2\xi>0$, where $\xi$ depends on $d,\tau,\kappa,M_i,s,|\beta|_1$.
\end{itemize}

Let $\rho'=\rho-s-\delta<\rho-s$. Let $\epsilon$ be such that
\[
\epsilon\leq\epsilon^*\delta^{4\tau},
\]
where $\epsilon^*$ depends on $d,\tau,\kappa,M,s,|\beta|_1$ and will be specified in the proof.

Then,
there exist functions $v^*,c^*\in\mathscr{A}_{\rho'}$ and numbers
$\sigma^*,\lambda^*\in\mathbb{R}$ such that
\begin{equation}\label{distance}
\begin{split}
& \mathscr{E}[v^*,\sigma^*,\lambda^*]=0,\\
& \mathscr{F}[v^*,\sigma^*,c^*]=0.
\end{split}
\end{equation}

In addition,
\begin{equation}\label{distance2}
\begin{split}
& || v^0 - v^* ||_{\rho'} \le C\delta^{-4\tau}\epsilon,\\
& | \lambda^0 - \lambda^* |
\le C \delta^{-2\tau}\epsilon,\\
& | \sigma^0 - \sigma^* |
\le C \delta^{-2\tau}\epsilon,\\
& || c^0 - c^* ||_{\rho'} \le C \delta^{-4\tau} \epsilon,
\end{split}
\end{equation}
where $C$ depends on $d,\tau,\kappa,M,s,|\beta|_1$. In the rest of the paper, we will use $C$ to denote any constant depending on these parameters.

We also have local uniqueness: Suppose $(v_1,\sigma_1,\lambda_1,c_1)$ and $(v_2,\sigma_2,\lambda_2,c_2)$ satisfying
  \begin{align}
    &\mathscr{E}[v_1,\sigma_1,\lambda_1]=\mathscr{E}[v_2,\sigma_2,\lambda_2]=0,\\
    &\mathscr{F}[v_1,\sigma_1,c_1]=\mathscr{F}[v_2,\sigma_2,c_2]=0,
  \end{align}
and the normalization condition
\begin{equation}
  \int_{\mathbb{T}^{d-1}}v_1d\psi=\int_{\mathbb{T}^{d-1}}v_2d\psi=0.
\end{equation}
If
\begin{equation}
  \max\{||v_2-v_1||_{3\rho/2},|\sigma_2-\sigma_2|,|\lambda_2-\lambda_1|\}<\epsilon^*{\rho},
\end{equation}
then, we have
\begin{equation}
  (v_1,c_1,\sigma_1,\lambda_1)=(v_2,c_2,\sigma_2,\lambda_2).
\end{equation}
Finally, we have that the solution depends on the parameter $\eta$ analytically.
\end{theorem}

The Theorem \ref{analytic KAM theorem}  is in an ``a-posteriori" format. Given an approximate solution which satisfies some non-degeneracy condition, we establish that there is a true solution nearby.

For the experts in KAM theory, we note that we have two parameters $s,\delta$ measuring the domain loss. The parameter $s$ measures the domain loss in the first step and the parameter $\delta$ measures the domain loss in the subsequent steps. The first step sets up several quantities which affect subsequent steps and are estimated in a different way from the other steps. The steps after the first are all very similar, and the estimates are very similar to the standard Nash-Moser scheme, even if the iterative step is very different.

To produce the approximate solution, we can use a variety of methods.

In the case that $W$ is small, we can take as approximate solution $(v^0, c^0, \sigma^0, \lambda^0) = (0,1,0,0)$ (the solution corresponding to $W=0$) and then (H1) just becomes smallness condition on $W$.

We can use the result of a Lindstedt series as approximate solutions. Then, we obtain a validation of the Lindstedt procedure, and a consequence for the (complex) differentiability of the solutions and the convergence of Lindstedt series (see Section \ref{sec:extensions}). See also \cite{SuZL15}.

Even if we will not discuss it in this paper, one could take as an approximate solution the outcome of a numerical computation and Theorem~\ref{analytic KAM theorem} gives a validation of the numerical result (provided that we check a few ``condition numbers"). Indeed, the proof leads to an efficient algorithm.

\section{Proof of Theorem \ref{analytic KAM theorem}}\label{sec:proof}
\subsection{Outline of the proof}\label{sec:outline of the proof}
The proof of Theorem \ref{analytic KAM theorem} is based on a quadratically convergent iterative method. The convergence of the method will be established by a Nash-Moser argument. Here we give an outline of the proof.

The linearization of the equilibrium equation is a linear second order difference equation. We will solve this linearized equilibrium equation using that it factorizes into two first order difference equations. In our case, we cannot write down the factorization in closed form (as in \cite{SuL1,Rafael'08}). So we will impose an auxiliary equation for the coefficients of the factorized equation. We will call this equation  the factorization equation. We will not need to solve exactly the factorization equation at each step, but we will require that there is an approximate solution of the factorization equation with an accuracy comparable to the accuracy of the approximate solution of the equilibrium equation. Then, we will derive an iterative method that improves both of the equilibrium equation and the factorization equation. This is reminiscent of the procedure to study an elliptic torus and its reducibility at the same time. See \cite{Moser'67, Poschel89}. Note without the parameter $\sigma$, we can not find coefficients satisfying the factorization equation.

Even with this novelty, we still have difficulty that the equilibrium equation and the factorization equation are coupled. We do not know how to establish convergence if we solve one equation after another. Therefore we will derive a way to solve the two coupled linearized equations simultaneously. After we factorize the linearized equilibrium equation, we will face three twisted cohomology equations (difference equations with non-constant coefficients), two for the linearized equilibrium equation and one for the linearized factorization equation. We will solve them using the technique in Section \ref{sec:motivation}. This completes one step of iteration.

One small technical complication is that the first step is estimated in a different way from the subsequent steps.
In the first step of the iteration, we compute some auxiliary function related to factorization (the auxiliary functions $\gamma$ related to the twisted cohomology equation) and estimate them from the original data. In subsequent steps, even if we compute the same auxiliary functions, we do not estimate them from the data but we estimate the change on the auxiliary functions induced by the (rather small) changes in the approximate solution.

Under some minor non-degeneracy conditions, we can repeat the process indefinitely. Finally, we will prove that the iterative procedure converges with suitably chosen domains under the analytic norm. This is very standard in KAM theory.

The local uniqueness will be obtained in Section \ref{sec:uniqueness} by showing that the linearized equation admits unique solutions using a standard technique. See \cite{CallejaL'10}.

We also note that the main part of Theorem \ref{analytic KAM theorem} has what is called ``a-posteriori" format in numerical analysis. We show that if there is an approximate enough solution that satisfies some non-degeneracy conditions, then, there is a true solution nearby. It is well known that this a-posteriori format leads to smooth dependence on parameters, bootstrap of regularity and several other consequences.

Finally, the analyticity with respect to parameters(we will use either $\epsilon$ or $\eta$ as parameters) will be a corollary of the existence of the perturbative expansions (see Section \ref{sec:perturbative}). Using the a-posteriori format of  Theorem \ref{analytic KAM theorem} and the local uniqueness, it follows that the solution is complex differentiable in the parameters. Hence, it is analytic on the parameters. Furthermore, since the perturbative expansions are the Taylor expansions of an analytic function they converge.

We note that the iterative process described above leads to an efficient algorithm, which we have formulated in Section \ref{sec:algorithm}.

\subsection{Motivation for the iterative step}\label{sec:motivation}
Our goal in this section is to devise a procedure that given an approximate solution produces another approximate 
solution with much smaller error and not much worse non-degeneracy conditions. This procedure is done for the equilibrium and factorization equations simultaneously.

\subsubsection{The equilibrium equation}
We consider the initial guess $(v^0,\sigma^0, \lambda^0)$ which solves equation \eqref{external force} with a small error $e$, where $||e||_{\rho}<\epsilon$, i.e.
\begin{equation}
  \label{eq:inv}
  v^0_++v^0_--2v^0+W((\psi,\eta)+\beta v^0)+\sigma^0v^0 + \lambda^0=e.
\end{equation}

The Newton procedure for the approximate solutions of the equilibrium equation \eqref{eq:inv} requires to find an update $(\hat{v},\hat{\sigma},\hat{\lambda})$ satisfying
\begin{equation}
  \label{eq:inv_lin}
  \hat{v}_++\hat{v}_--2\hat{v}+\partial_{\beta}W_{v^0}\hat{v}+\hat{\sigma}v^0+\sigma^0\hat{v}+\hat{\lambda}=-e,
\end{equation}
where we denote $\partial_{\beta}W=\beta\cdot \nabla W$ for simplicity and we use $W_{v}$ to indicate $W((\psi,\eta)+\beta v)$. 

\subsubsection{The factorization equation}
The equation \eqref{eq:inv_lin} is not easy to solve directly. Nevertheless, we impose that it can be factorized into two first order difference equations with non-constant coefficients. If we accomplish this, we can solve \eqref{eq:inv_lin} using the theory of twisted cohomology equations developed in Section \ref{sec:twisted}.

Therefore we want that the operator $\mathscr{L}$ (recall that we denote $\mathscr{L}=T_{\Omega}+T_{-\Omega}-2+\partial_{\beta}W_{v^0}+\sigma$) is factorized into two first order operators:
\begin{equation}    \label{original equilibrium equation}
  \mathscr{L}=\mathscr{A}_+\mathscr{A}_-,
\end{equation}
where
\begin{align}
  \label{eq:firstordereq}
  \mathscr{A}_+&=a(\psi)T_{\Omega}-b(\psi)\\
  \mathscr{A}_-&=c(\psi)-d(\psi)T_{-\Omega}
\end{align}
and $a,b,c,d$ are the new unknowns.

A direct calculation shows
\begin{equation}
  \mathscr{A}_+\mathscr{A}_-\hat{v}=a(\psi)c_+(\psi)\hat{v}_+-[a(\psi)d_+(\psi)+b(\psi)c(\psi)]\hat{v}+b(\psi)d(\psi)\hat{v}_-.
\end{equation}

Hence we want to choose $a,b,c,d$ satisfying the following equations
\begin{align}\label{eq:coeff}
  a(\psi)c_+(\psi)&=1\nonumber,\\
  -[a(\psi)d_+(\psi)+b(\psi)c(\psi)]&=-2+\partial_{\beta}W_{v^0}+\sigma^0,\\
  b(\psi)d(\psi)&=1.\nonumber
\end{align}

Note that the problem of factorization has always many solutions. If $\mathscr{A}_+\mathscr{A}_-$ is a solution and $g$ is any invertible function, then $\tilde{\mathscr{A}}_+\tilde{\mathscr{A}}_-$ is also a solution, where
  \begin{equation}
    \label{eq:change}
    \begin{split}
      \tilde{\mathscr{A}}_+&=\mathscr{A}_+g,\\
      \tilde{\mathscr{A}}_-&=g^{-1}\mathscr{A}_-.
    \end{split}
  \end{equation}
We can use this non-uniqueness to impose some extra normalization condition. In this paper, we will  take the normalization
\begin{equation}
  \label{eq:normalized coeff}
  b(x)=1.
\end{equation}

With condition \eqref{eq:normalized coeff} we can simplify the system (\ref{eq:coeff}) to (after eliminating $a$):
\begin{equation}
  (-c+2-\partial_{\beta}W_{v^0}-\sigma^0)c_+=1.
\end{equation}
This is how the operator $\mathscr{F}$ comes into play. Note it is a non-linear, non-local equation which we will have to solve iteratively.
By assumption (H1), we can solve this equation with an error $f$ using some initial guess $c^0$, i.e.
\begin{equation}\label{eq:factorization}
  (-c^0+2-\partial_{\beta}W_{v^0}-\sigma^0)c^0_+-1=f
\end{equation}
and we have $||f||_{\rho}<\epsilon$.

Then the Newton procedure for \eqref{eq:factorization} requires solving the following equation for $\hat{c}$:
\begin{equation}
  \label{eq:F}
  -c^0_+\hat{c}+(-c^0+2-\partial_{\beta}W_{v^0}-\sigma^0)\hat{c}_+-c^0_+\hat{\sigma}-\partial_{\beta}\partial_{\beta}W_{v^0} c^0_+\hat{v}=-f.
\end{equation}
Then $c+\hat{c}$ would be a more accurate solution.

We replace $\mathscr{L}$ with $\mathscr{A}_+\mathscr{A}_-$ in the original equilibrium equation \eqref{eq:inv_lin} to get
\begin{equation}
  \label{eq:I}
  \mathscr{A}_+\mathscr{A}_-\hat{v}+\hat{\sigma}v^0+\hat{\lambda}=-e.
\end{equation}

\begin{Remark}
  Note that the equation \eqref{eq:I} is slightly different from \eqref{eq:inv_lin}, so the $(\hat{v},\hat{\sigma},\hat{\lambda},\hat{c})$ we find is not exactly the solution of the linearized equation of \eqref{external force}. But the only difference between \eqref{eq:I} and \eqref{eq:inv_lin} is a term $f\hat{v}(c_+^0)^{-1}$, which we will show is quadratic in $e,f$ (since $\hat{v}$ will be of the order of $e$). 
\end{Remark}
\begin{Remark}
To obtain quadratic convergence, the full Newton step will have to improve both the equilibrium and the factorization equation at the same time. (A moment's reflection shows that improving one equation and then the other does not lead to quadratic convergence.)
We will prove our Newton procedure using \eqref{eq:I} instead of \eqref{eq:inv_lin} and improving the factorization still converges quadratically to a solution of \eqref{external force}.
\end{Remark}

In the following section, we will develop a method to solve \eqref{eq:I} and \eqref{eq:F} simultaneously. In Section \ref{sec:estimation of errors} we will obtain estimates for the size of the error. Then, we will show in Section \ref{sec:convergence} that they lead to an improved solution and that the process can be iterated and converges to a solution.

\subsection{Solving the linearized equations}\label{sec:solve the linearization}
Our goal in this section is to solve \eqref{eq:I} and \eqref{eq:F} simultaneously.

The difficulty arises because the equations \eqref{eq:I} and \eqref{eq:F} are coupled (the unknown $\hat{\sigma}$ appears in both of them). Observe that since $\hat{\sigma}$ appears in an affine way, we can guess that the solutions will be affine in $\hat{\sigma}$. This allows to uncouple the equations.

When we find the corrections $(\hat{v},\hat{\sigma},\hat{\lambda},\hat{c})$, we can write down the updated solution of the equilibrium equation \eqref{external force} and the factorization equation \eqref{eq:factorization}:
\begin{equation}\label{eq:improved solution}
(v^1,\sigma^1,\lambda^1,c^1)=(v^0+\hat{v},\sigma^0+\hat{\sigma},\lambda^0+\hat{\lambda},c^0+\hat{c}),
\end{equation}
which can be used in the next iterative step. We will show that the improved solution \eqref{eq:improved solution} is indeed more approximate (in a smaller domain). Then the procedure can be repeated indefinitely and converges to a solution if the errors are small enough.

Now we write
\begin{align}
  \label{eq:v_affine}
  \hat{v}&=A+\hat{\sigma}B,\\
  \hat{\lambda}&=G+\hat{\sigma}D,\label{eq:lambda_affine}
\end{align}
where $A$ and $B$ are functions, $G$ and $D$ are numbers. All these quantities will be determined shortly.

Then the equation (\ref{eq:I}) becomes
\begin{align}
  \label{eq:cohom_inv1}
  \mathscr{A}_+\mathscr{A}_-A+G&=-e,\\
  \label{eq:cohom_inv2}
  \mathscr{A}_+\mathscr{A}_-B+D&=-v^0.
\end{align}
The Newton equation for factorization \eqref{eq:F}  after substitution \eqref{eq:cohom_inv1} and \eqref{eq:cohom_inv2} becomes
\begin{equation}
   \label{eq:cohom_fac}
  -c^0_+\hat{c}+(-c^0+2-\partial_{\beta}W_{v^0}-\sigma^0)\hat{c}_++(-c^0_+-\partial_{\beta}\partial_{\beta} W_{v^0} c_+^0B)\hat{\sigma}=\partial_{\beta}\partial_{\beta}W_{v^0} c^0_+A-f.
\end{equation}


Hence the Newton step improving simultaneously the invariance and factorization equations consists in solving \eqref{eq:cohom_inv1}, \eqref{eq:cohom_inv2}, \eqref{eq:cohom_fac}.
Notice that this system has an upper triangular structure, we can solve \eqref{eq:cohom_inv1}, \eqref{eq:cohom_inv2} for $A,B, G, D$ and then solve \eqref{eq:cohom_fac} for $\hat{c}$ and $\hat{\sigma}$. Once we have $\hat{\sigma}$, we can determine $\hat{v}$ and $\hat{\lambda}$ using \eqref{eq:v_affine} and \eqref{eq:lambda_affine}.

To solve \eqref{eq:cohom_inv1} and \eqref{eq:cohom_inv2}, we observe that each of them can be obtained by solving two cohomology equations with non-constant coefficients.
We can find the zero average solution of \eqref{eq:cohom_inv1} and \eqref{eq:cohom_inv2}. In fact, $G$ is determined so that $A$ has zero average and $D$ is determined so that $B$ has zero average. This ensures that 
$\langle \hat{v} \rangle =0$.

We note that the choice of the constants in the solution of two consecutive twisted cohomology equations is a simple extension of the arguments developed at the end of Section \ref{sec:twisted}.

In the case that $\mathscr{A}_+,\mathscr{A}_-$ both have different average coefficients, we see that the average of the solution as a function of $G$ is the composition of two affine functions, hence affine and if the linear part of each of them is not zero, we obtain that the composition of the two affine functions has non-zero derivative. 

In the case that $\mathscr{A}_+$ has equal average coefficient, but $\mathscr{A}_-$ does not, we choose $G$ so that $\mathscr{A}_+$ is solvable and choose the average of its solution so that the solution of $\mathscr{A}_-$ satisfies the normalization.

In the case that $\mathscr{A}_+$ has different average coefficient but $\mathscr{A}_-$ has the same average coefficient, we choose $G$ so that the solution produced by $\mathscr{A}_+$ satisfies the compatibility conditions for $\mathscr{A}_-$. Then, the solution of $\mathscr{A}_-$ is defined up to a constant which can be chosen in a unique way to satisfy the normalization condition.

In the case that both $\mathscr{A}_+,\mathscr{A}_-$ have the same average coefficients--which is the case that appears in standard KAM theory--we choose $G$ so that $\mathscr{A}_+$ is solvable. Then, we use the free parameter of the solution of $\mathscr{A}_-$ to ensure the solvability of $\mathscr{A}_-$ and use the free constant in the solution of $\mathscr{A}_-$ to adjust the normalization.

Therefore, we obtain the solution of \eqref{eq:cohom_inv1}. Similar procedure can be done for \eqref{eq:cohom_inv2}.

Now that we have solved \eqref{eq:cohom_inv1} and \eqref{eq:cohom_inv2} we turn to solving \eqref{eq:cohom_fac}.
As long as
\begin{equation}\label{transversal to factorization manifold}
 \int_{\mathbb{T}^{d-1} } \left(-c_+^0-\partial_{\beta}\partial_{\beta} W_{v^0} c_+^0B\right)\ d\psi \neq 0,
\end{equation}
which can be proved under some non-degeneracy conditions on the initial guesses, equation (\ref{eq:cohom_fac}) is a twisted cohomology equation. Provided that $-c^0_+$ and $-c^0+2-\partial_{\beta}W_{v^0}-\sigma^0$ are away from zero, we can apply Lemma \ref{lm:twisted} to solve (\ref{eq:cohom_fac}) after $A$ and $B$ are found.

To make the procedure clear, we will write the full algorithm in Section~\ref{sec:algorithm}. In Section \ref{sec:estimation}, we will carefully estimate the errors after one iterative step to show that the errors of the equilibrium and the factorization equation are both reduced quadratically after performing the corrections in the iterative step. Finally, we will prove the convergence rigorously in Section~\ref{sec:convergence}.
\subsection{Formulation of the iterative step}
\label{sec:algorithm}
\begin{alg}
  \begin{enumerate}
  \item [(1)] Given initial guesses $v^0,\sigma^0, \lambda^0, c^0$, set $a^0=(c^0_+)^{-1}$.
  \item [(2)] Calculate the errors $e=\mathscr{E}(v^0,\sigma^0,\lambda^0)$ and $f=\mathscr{F}(v^0,\sigma^0,c^0)$.
  \item [(3)]  Find $\gamma_{a^0}$ and $\bar{a}^0$ satisfying 
    \begin{equation*}
      \log a^0=\log \bar{a}^0+\log(\gamma_{a^0})_+-\log(\gamma_{a^0}).
    \end{equation*}
    Also find $\gamma_{c^0}$ and $\bar{c}^0$ satisfying 
    \begin{equation*}
      \log c^0=\log\bar{c}^0+\log(\gamma_{c^0})-\log(\gamma_{c^0})_-
    \end{equation*}

   \item [(4)] Compute $\bar{a}_+, \bar{b}_+, \gamma_{a_+}, \gamma_{b_+}$ (the average coefficients and the auxiliary functions) for $\mathscr{A}_+$ and $\bar{a}_-, \bar{b}_-, \gamma_{a_-}, \gamma_{b_-}$ for $\mathscr{A}_-$.

   \item [(5a)] If $\bar{a}_+\neq\bar{b}_+, \bar{a}_-\neq \bar{b}_-$, compute 
\[
\alpha= -\mathscr{A}_+^{-1} \mathscr{A}_-^{-1} e,\
\beta= -\mathscr{A}_+^{-1} \mathscr{A}_-^{-1} 1.
\]
                Set $G=-\frac{\langle \alpha \rangle}{\langle \beta \rangle}$. Set $A= \alpha + G\beta$.

   \item [(5b)] If $\bar{a}_+=\bar{b}_+, \bar{a}_-\neq \bar{b}_-$, choose $G= -\langle e\rangle$.
                 Set $\tilde{\alpha} = - \mathscr{A}_+^{-1}(e+ G)$ and $\langle \tilde{\alpha} \rangle =0$.
                 Set $\beta_1 = \mathscr{A}_-^{-1}\tilde{\alpha}, \beta_2= -\mathscr{A}_-^{-1} 1$. 
                 Set $A= \beta_1 - \beta_2 \frac{\langle \beta_1 \rangle}{\langle \beta_2 \rangle}$.

   \item [(5c)] If $\bar{a}_+\neq\bar{b}_+, \bar{a}_-= \bar{b}_-$, choose $\alpha_1 = - \mathscr{A}_+^{-1} e, \alpha_2 = - \mathscr{A}_+^{-1} 1$. 
Set $G= \frac{\langle \alpha_1 \rangle}{\langle \alpha_2 \rangle}$. 
Set $A=\mathscr{A}_-^{-1}(\alpha_1 - G\alpha_2)$ and $\langle A\rangle =0$.

   \item [(5d)]  If $\bar{a}_+ = \bar{b}_+, \bar{a}_-= \bar{b}_-$, choose $G= - \langle e\rangle$.
Set $\alpha=\mathscr{A}_+^{-1} (e- \langle e\rangle)$ and $\langle \alpha\rangle=0$.
Set $A= \mathscr{A}_-^{-1} \alpha$ and $\langle A\rangle =0$.

   \item [(6)]  Find $B$ and $D$ in a similar way as we found $A$ and $G$. Also set $\langle B\rangle=0.$

   \item [(7)]  Find $\hat{\sigma}$ and $\hat{c}$ by solving (\ref{eq:cohom_fac}).
   \item [(8)]  Set $\hat{v}=A+\hat{\sigma}B$ and $\hat{\lambda}=G+\hat{\sigma}D$.
   \item [(9)]  Update $u,\lambda, \sigma$ and $c$ and repeat the steps.
  \end{enumerate}
\end{alg}

\begin{Remark}
  When we apply repeatedly the iterative steps to obtain estimates, it would be advantageous in Step 3 to use the information we have on $\gamma_a$ computed in the previous steps, because it leads to better estimates.
In the first step, we can only use Cauchy estimates on the initial data. See Section \ref{sec:outline of the proof}.
\end{Remark}

\begin{Remark}
The iterative method described above achieves
quadratic convergence (see Section \ref{sec:convergence}). It only entails performing algebraic operations among functions, composing them,
taking derivatives and solving the cohomology equations.

If we discretize a function by the values of the point in a grid of $N$ points and
 (redundantly) $N$ Fourier coefficients,
we note that each of the operations above requires $N$ operations either in the
grid representation or in the Fourier space representation. If we obtain either the Fourier or
the real space representation for a function, we can obtain the other representation using the
FFT algorithm that requires $N\log(N)$ operations.

Hence the method in Section~\ref{sec:proof} achieves quadratic convergence but no matrix inversion (or storage) is
required. It only requires $O(N)$ storage and $O(N \log(N))$ operations per step.

This has already been observed in \cite{CallejaL09} in the periodic case
(both for short range and for long range interactions)  and \cite{SuL1}
for the quasi-periodic non-resonant case. Numerical implementations in the
non-resonant periodic case were carried out in \cite{BlassL}.
We have not implemented the above algorithm, but we think it would be interesting to do so. 
\end{Remark}

\subsection{Estimates on the corrections}
\label{sec:estimation}
Denote $V^0=-c^0+2-\partial_{\beta}W_{v^0}-\sigma^0$. In the first step, we argue as follows. From \eqref{estimates on log}, it is clear that all the quantities $||\gamma_{c^0}||_{\rho-\delta/4},||\gamma^{-1}_{c^0}||_{\rho-\delta/4},||\gamma_{(c^0)^{-1}}||_{\rho-\delta/4},$ $||\gamma^{-1}_{(c^0)^{-1}}||_{\rho-\delta/4},||\gamma_{V^0}||_{\rho-\delta/4}$ and $||\gamma_{V^0}^{-1}||_{\rho-\delta/4}$ are bounded by some constant which depends only on $d,\tau,\kappa,M,s$. We denote the bound by $E$.

In subsequent steps, we will just assume that we have the bounds $E$ for the above auxiliary quantities. These bounds will be derived from the bounds in the first step.

Apply the estimate \eqref{twist final estimate} of twisted cohomology equations for equations (\ref{eq:cohom_inv1}) and (\ref{eq:cohom_inv2}), also recall the averages of $A$ and $B$ both vanish we have
\begin{align}
  ||A||_{\rho-\delta/2}&\le CE^4\delta^{-2\tau}||\tilde{e}||_{\rho}=C\delta^{-2\tau}||\tilde{e}||_{\rho},\\
  ||B||_{\rho-\delta/2}&\le CE^4\delta^{-2\tau}||\tilde{v}^0||_{\rho}=C\delta^{-2\tau}||\tilde{v}^0||_{\rho},\\
  |G| & \le CE^4||e||_{\rho}\le C\epsilon,\\
  |D| & \le CE^4||v^0||_{\rho}\le C.
\end{align}

Choose $m$ small, we can make $||B||_{\rho-\delta/2}<1$. Then
\begin{equation}
  ||(-c^0_+-\partial_{\beta}\partial_{\beta}W_{v^0}c^0_+B)^{-1}||_{\rho-\delta/4}\le (1-M_1)^{-1}(1-M_2)^{-1}.
\end{equation}
This guarantees $\int_{\mathbb{T}^{d-1}} \left(\right.-c^0_+-\partial_{\beta}  \partial_{\beta} W_{v^0}c^0_+B \left.\right) d\psi\neq 0$. Therefore the cohomology equation (\ref{eq:cohom_fac}) is solvable.

The estimate for (\ref{eq:cohom_fac}) gives
\begin{equation}\label{eq:estimate for sigma_hat}
  \begin{split}
    |\hat{\sigma}|&\le C(||A||_{\rho-\delta/2}+||f||_{\rho})\\
    &\le C\delta^{-2\tau}\epsilon+C\epsilon\\
    &\le C\delta^{-2\tau}\epsilon,
  \end{split}
\end{equation}
and
\begin{equation}\label{estimate for c_hat}
  \begin{split}
  ||\hat{c}||_{\rho-\delta}&\le CE^4\delta^{-\tau}(||A||_{\rho-\delta/2}+||f||_{\rho})\\
  &\le C\delta^{-3\tau}\epsilon.
  \end{split}
\end{equation}

From (\ref{eq:v_affine}), we have the estimate
\begin{equation}\label{eq:estimate for v_hat}
  \begin{split}
  ||\hat{v}||_{\rho-\delta/2}&\le ||A||_{\rho-\delta/2}+|\hat{\sigma}|||B||_{\rho-\delta/2}\\
  &\le C\delta^{-4\tau}\epsilon.
  \end{split}
\end{equation}

Also, from the estimates for $G,D$ and $\sigma$, we have
\begin{equation}\label{estimate for v_hat}
  |\lambda|\le C\epsilon+C\delta^{-2\tau}\epsilon\le C\delta^{-2\tau}\epsilon .
\end{equation}

\subsection{Estimates on the improved error}\label{sec:estimation of errors}
Next we estimate the new error for both equilibrium and factorization equation after one iteration step.

We need the Assumption (H4) to make sure the compositions can be done in appropriate domains. Recall that $\|W_{v^0}\|_{\mathscr{D}}<M$. Also note $\hat{v}$ is relatively smaller than $v^0$. From \eqref{estimate for v_hat}, we see that as long as $\epsilon$ is small enough the range of $v+\hat{v}$ is inside the domain of $W$. So we can choose proper $\xi$ such that $\dist(\mathscr{R}_{v^0},\partial\mathscr{D})\geq \xi>0$.

For the subsequent iterative step, we note that if $\delta^{-4\tau}\epsilon$ is small enough, $||\hat{v}||_{\rho-\delta}$ will also be small. So that under smallness condition on $\delta^{-4\tau}\epsilon$, we can ensure that $v+\hat{v}$ is well inside the domain of definition of $W$. The linear estimates are valid for all $\delta$'s, but in order to ensure that we can apply the non-linear estimates, we have to choose $\delta$'s (the domain loss) in such a way that $\delta^{-4\tau}\epsilon$ is small.

It is standard in KAM theory (and we will do it later in section \ref{sec:convergence}) that one can choose domain losses $\delta_n$ in such a way that the composition condition is met, such that  $\delta_n$'s go to zero fast enough, so there is still a domain left. Therefore, the composition will remain in the proper domain for all iterative steps and the procedure will converge in a non-trivial domain.

Using the Taylor expansions (see Proposition \ref{composition}) and the equations \eqref{eq:inv} and \eqref{eq:factorization} for the initial guesses, we have:
\begin{equation}
\begin{split}
      \mathscr{E}[v^0+&\hat{v},\sigma^0+\hat{\sigma},\lambda^0+\hat{\lambda}] \\
    =& \mathscr{E}[v^0,\sigma^0,\lambda^0]+\hat{v}_++\hat{v}_--2\hat{v}+\hat{\sigma}v^0+\sigma^0\hat{v}+\hat{\sigma}\hat{v}+\hat{\lambda}-W_{v^0}+W_{v^0+\hat{v}}\\
    =& e+\mathscr{A}_+\mathscr{A}_-\hat{v}-f\hat{v}-(\partial_{\beta}W_{v^0}+\sigma^0)\hat{v}+\hat{\sigma}v^0+\sigma^0\hat{v}+\hat{\sigma}\hat{v}+\hat{\lambda}\\
     &-W_{v^0}+W_{v^0+\hat{v}}\\
    =& -f\hat{v}+\hat{\sigma}\hat{v}+(W_{v^0+\hat{v}}-W_{v^0}-\partial_{\beta}W_{v^0}\hat{v}),
\end{split}
\end{equation}
and
\begin{equation}
\begin{split}
     \mathscr{F}[v^0+&\hat{v},\sigma^0+\hat{\sigma},c^0+\hat{c}] \\
  = & \mathscr{F}[v^0,\sigma^0,c^0]+[-(c^0+\hat{c})+2-\partial_{\beta}W_{v^0}-(\sigma^0+\hat{\sigma})]\hat{c}_+\\
&+[-\hat{c}-\partial_{\beta}W_{v^0+\hat{v}}+\partial_{\beta}W_{v^0}-\hat{\sigma}]c^0_+\\
  = & [-\hat{c}-(\partial_{\beta}W_{v^0+\hat{v}}-\partial_{\beta}W_{v^0})-\hat{\sigma}]\hat{c}_+\\
&-[(\partial_{\beta}W_{v^0+\hat{v}}-\partial_{\beta}W_{v^0})c^0_+-\partial_{\beta}\partial_{\beta}W_{v^0}c^0_+\hat{v}].
\end{split}
\end{equation}

Take $0<\delta<\rho$, we have
\begin{align}
  ||\mathscr{E}[v^0+\hat{v},&\sigma^0+\hat{\sigma},\lambda^0+\hat{\lambda}]||_{\rho-\delta}&\\\nonumber
\le & ||f\hat{v}||_{\rho-\delta}+||\hat{\sigma}\hat{v}||_{\rho-\delta}+||W_{v^0+\hat{v}}-W_{v^0}-\partial_{\beta}W_{v^0}\hat{v}||_{\rho-\delta}\\\nonumber
\le & ||f||_{\rho-\delta}||\hat{v}||_{\rho-\delta}+||\hat{\sigma}||_{\rho-\delta}||\hat{v}||_{\rho-\delta}+C||\hat{v}||^2_{\rho-\delta}.
\end{align}

Also,
\begin{align}
  ||\mathscr{F}[v^0+\hat{v},&\sigma^0+\hat{\sigma},c^0+\hat{c}]||_{\rho-\delta}&\\\nonumber
\le & (||\hat{c}||_{\rho-\delta}+C||\hat{v}||_{\rho-\delta}+||\hat{\sigma}||_{\rho-\delta})||\hat{c}||_{\rho-\delta}+C||\hat{v}||^2_{\rho-\delta}||c^0||_{\rho-\delta}.
\end{align}

Now it is clear the new errors are quadratic in $f,\hat{v},\hat{\sigma},\hat{c}$, multiplied by the domain loss to a negative power. This is what are called ``tame estimates'' in KAM theory.

The final estimates for both equilibrium and factorization equations on the updated solutions are
\begin{equation}
  ||\mathscr{E}[v^0+\hat{v},\sigma^0+\hat{\sigma},\lambda^0+\hat{\lambda}]||_{\rho-\delta}\le C\delta^{-8\tau}\epsilon^2
\end{equation}
and
\begin{equation}
  ||\mathscr{F}[v^0+\hat{v},\sigma^0+\hat{\sigma},c^0+\hat{c}]||_{\rho-\delta}\le C\delta^{-8\tau}\epsilon^2.
\end{equation}

Therefore we get updates for $c,v,\sigma$ and $\lambda$ which reduce the error of both the equilibrium equation and the factorization equation quadratically under the condition that all the bounds we derived before still hold for all iterative steps. We can prove the convergence following standard procedure, which we give in the following.

\subsection{Proof of convergence}
\label{sec:convergence}
The main effect of the iterative step is to reduce the error in \eqref{eq:I} and \eqref{eq:F}. But it could also deteriorate the constants in the non-degeneracy assumptions when it improves the solution.

Therefore the first goal of this section is to show that the constants in non-degeneracy assumptions of the solution do not deteriorate much and that the deterioration can be estimated by the error. In this case, as we will see, the convergence can be proved by choosing suitable domains (this is the choice we will do first) for each iterative step as is standard in KAM theory.

We will use 
subscript  $n$ to denote the quantities $\rho, \delta$ and $\epsilon$ after application of the iterative step $n$ times, while we use superscript $n$ for $v,c,\sigma,\lambda$ and $B$. We take
\begin{equation}
  \rho_0=\rho,\rho_1=\rho-s-\delta_0,\delta_n=\delta_0\cdot 2^{-n} \quad \text{and}\quad \rho_{n+1}=\rho_{n}-\delta_n.
\end{equation}

Denote $\epsilon_n=||\mathscr{E}(v^n,\delta^n,\lambda^n)||_{\rho_n}$ and $\tilde{\epsilon}_n=||\mathscr{F}(v^n,\sigma^n,c^n)||_{\rho_n}$. We will prove the following holds for all iterative steps by induction.
\begin{description}
\item[(B1)] All the quantities $||v^n||_{\rho_n},||c^n-1||_{\rho_n},|\sigma^n|,||W_{v^0}||_{\rho_n},||B^n||_{\rho_n},||\gamma_{c^n}||_{\rho_n},$\\$||\gamma_{c^n}^{-1}||_{\rho_n},||\gamma_{V^n}||_{\rho_n},||\gamma_{V^n}^{-1}||_{\rho_n}$, and $||(-c^0_+-\partial_{\beta}\partial_{\beta}W_{v^0}c^0_+B)^{-1}||_{\rho_n}$ are still bounded uniformly in $n$. To be precise, for any quantity $A^0<E$, we will prove $|A^n-A^{n-1}|<E\cdot2^{-n} ) $, therefore, $A^n<2E$ for any $n$.
\item[(B2)] Denote $\mathscr{R}_{v^n}=\{(\psi,\eta)+\beta v^n,(\psi,\eta)\in \mathbb{T}^d_{\rho_n}\}$. Then $\dist(\mathscr{R}_{v^n},\partial\mathscr{D})$ is also bounded by $\xi^*$.
\item[(B3)] $\epsilon_{n+1}\le (C\epsilon_0)^{2^{n+1}},\quad \tilde{\epsilon}_{n+1}\le (C\tilde{\epsilon}_0)^{2^{n+1}}$.
\end{description}
Note that as a consequence of (B3), and the choices of $\delta_n$, we obtain that the composition assumption holds
if $\epsilon_0$ is small enough.

The first step is already shown  in Section \ref{sec:estimation}. Now we assume the first $n$ steps are proved.

We first prove (B1). We only prove the bound
\begin{equation}
  \label{eq:bounds}
||\gamma_{c^n}-\gamma_{c^{n-1}}||_{\rho_n-\delta/4}<E\cdot2^{-n}.
\end{equation}
The proof of the other bounds is similar (up to changing the symbols). 

From \eqref{estimates on log}, we have
\begin{equation}\label{estimate for log_sigma_c}
  ||\log \gamma_{c^{n}}-\log \gamma_{c^{n-1}}||_{\rho_{n}-\delta_n/4}\leq C\delta_n^{-\tau}||\log c^{n}-\log c^{n-1}||_{\rho_n}.
\end{equation}

From estimate \eqref{estimate for c_hat} for $\hat{c}$, we have
\begin{equation}
  ||\hat{c}^{n}||_{\rho_{n}}\le C\delta_n^{-3\tau}\epsilon_n.
\end{equation}

As a result, we have
\begin{equation}\label{estimate for log_c}
  ||\log c^{n}-\log c^{n-1}||_{\rho_{n}}\le C\delta_n^{-3\tau}\epsilon_n.
\end{equation}

Combining \eqref{estimate for log_sigma_c} and \eqref{estimate for log_c}, we have
\begin{equation}
  ||\log \gamma_{c^{n}}-\log \gamma_{c^{n-1}}||_{\rho_{n}-\delta_n/4}\le C\delta_n^{-4\tau}\epsilon_n.
\end{equation}
Therefore we have (using that $\gamma_{c^{n}}$ are uniformly bounded)
\begin{equation}
  ||\gamma_{c^{n}}-\gamma_{c^{n-1}}||_{\rho_{n}-\delta_n/4}\le C\delta_n^{-4\tau}\epsilon_n.
\end{equation}

Note for (B2), we only need bounds for $v^n$, which is easy to prove.

Now recall $\epsilon_n<(C\epsilon_0)^{2^n}$. Therefore if we choose $\epsilon_0$ such that $\delta^{-4\tau}\epsilon_0$ is small enough, we can guarantee \eqref{eq:bounds} (i.e. (B2)). (B3) can be shown as follows
\begin{equation}
  \begin{split}
    \epsilon_{n}&\le C\delta_{n-1}^{-8\tau}\epsilon^2_{n-1}\\
    &=C\delta_0^{-8\tau}(2^{8\tau})^{n-1}\epsilon_{n-1}^2\\
    &\le (C\delta_0^{-8\tau})^{1+2+\cdots+2^{n-1}}(2^{8\tau})^{(n-1)+(n-2)\cdot 2+\cdots+1\cdot (n-1)}\epsilon_0^{2^{n}}\\
    &\le (C\delta_0^{-8\tau}2^{8\tau}\epsilon_0)^{2^{n}}\\
    &= (C2^{8\tau}\epsilon_0)^{2^{n}}.
  \end{split}
\end{equation}
Similar estimates also hold for $\tilde{\epsilon}_n$.

\subsection{Proof of local uniqueness}\label{sec:uniqueness}
  Suppose we have two solutions $(v_1,\sigma_1,\lambda_1,c_1)$ and $(v_2,\sigma_2,\lambda_2,c_2)$ of both the equilibrium and the factorization equations, which also satisfy the non-degeneracy conditions. We have
  \begin{equation}
    \mathscr{E}[v_1,\sigma_1,\lambda_1]=\mathscr{E}[v_2,\sigma_2,\lambda_2]=0.
  \end{equation}

Then we can write
\begin{equation}\label{identity to prove uniqueness}
  \begin{split}
      \mathscr{E}[v_2,\sigma_2,\lambda_2]&=\mathscr{E}[v_1,\sigma_1,\lambda_1]+D\mathscr{E}[v_1,\sigma_1,\lambda_1]\cdot (v_2-v_1,\sigma_2-\sigma_1,\lambda_2-\lambda_1)+R^2\\
      &=D\mathscr{E}[v_1,\sigma_1,\lambda_1]\cdot (v_2-v_1,\sigma_2-\sigma_1,\lambda_2-\lambda_1)+R^2=0,
  \end{split}
\end{equation}
where $R^2$ is the Taylor remainder for $\mathscr{E}$.

Another way to read this identity \eqref{identity to prove uniqueness} is to say that $v_1-v_1,\lambda_2-\lambda_1,\sigma_2-\sigma_1$ is a solution of the equation
\begin{equation}
  \label{eq:rq:I_2}
    \mathscr{A}_+\mathscr{A}_-(v_2-v_1)+(\sigma_2-\sigma_1)v^1+(\lambda_2-\lambda_1)=-R^2.
\end{equation}
Similarly, we have 
\begin{equation}
  \label{eq:rq:F_2}
    -c^1_+(c_2-c_1)+(-c^1+2-\partial_{\beta}W_{v^1}-\sigma^1)(c_2-c_1)_+-c^1_+(\sigma_2-\sigma_1)-\partial_{\beta}\partial_{\beta}W_{v^1} c^1_+(v_2-v_1)=-\tilde{R}^2.
\end{equation}
We have shown in Section \ref{sec:estimation} that the solutions of \eqref{eq:rq:I_2} and \eqref{eq:rq:F_2} satisfying the normalization condition are unique and satisfy the bounds \eqref{eq:estimate for sigma_hat} and \eqref{eq:estimate for v_hat}. So we have
\begin{equation}
  ||v_2-v_1||_{\frac{\rho}{2}}+||c_2-c_1||_{\frac{\rho}{2}}+|\sigma_2-\sigma_1|\leq C\rho^{-4\tau}||R^2||_{\rho}\leq C\rho^{-4\tau}(||v_2-v_1||_{\rho}+||c_2-c_1||_{\rho}+|\sigma_2-\sigma_1|)^2 .
\end{equation}
Since $||v_2-v_1||_{\rho}$, $||c_2-c_1||_{\rho}$ and $|\sigma_2-\sigma_1|$ are all small, we have
\begin{equation}
  \begin{split}
    ||v_2-v_1||_{\frac{\rho}{2}}\leq &C \rho^{-4\tau} ||v_2-v_1||^2_{\rho}\\
     \leq & C \rho^{-4\tau} ||v_2-v_1||_{\frac{\rho}{2}}||v_2-v_1||_{\frac{3\rho}{2}}.
  \end{split}
\end{equation}
The last inequality uses the interpolation inequality \eqref{interpolation inequalities}.

Therefore, as long as $\rho^{-4\tau}||v_2-v_1||_{\frac{3\rho}{2}}$ is sufficiently small (depending on the properties of the auxiliary function $c$ in the factorization equation), we have $||v_1-v_2||_{\frac{\rho}{2}}=0$, which also implies $\sigma_2=\sigma_1$ and $\lambda_2=\lambda_1$.

\section{Consequence of Theorem~\ref{analytic KAM theorem} and its
proof}\label{sec:extensions}
\subsection{Perturbative series around any solution}
\label{sec:perturbative}

In this section we present methods to compute formal power
series expansions. In Section~\ref{sec:convergenceseries} we will show that, under Diophantine
conditions on the frequency, they converge on a sufficiently small domain.

Assume:
\begin{description}
\item [{\bf A0}] We  are given a family of interaction potentials $W^\mu$
indexed
by an external parameter $\mu$ which can  be complex.
We assume that $W^\mu$ is analytic in both its arguments $\psi,\eta$ and the parameter $\mu$.

\item
[{\bf A1}] For some parameter $\mu_0$ we have a solution
$v^0, \sigma^0, \lambda^0 $ of $\eqref{external force}$.
We assume that $v^0 \in \mathscr{A}_\rho$ for some $\rho > 0$.

\item
[{\bf A2}] The operator $\mathscr{L}_{v^0}$, the linearization 
of the equilibrium equation  factorizes at $\mu_0$ into two first order operators. 

\end{description}

Our goal is to find a perturbative expansion of the solutions of \eqref{external force} in formal
power series of $\mu - \mu_0$. Later, we will show
that these perturbative expansions are convergent following
an argument of \cite{Moser'67} which is made much easier by the
a-posteriori format of Theorem~\ref{analytic KAM theorem}.  See Section~\ref{sec:convergenceseries}.
We seek
\begin{equation}\label{generalexpansion}
\begin{split}
& v^\mu = v^0 + \sum_{n > 0} (\mu - \mu_0)^n v^n \\
& \sigma^{\mu} = \sigma^0+\sum_{n > 0}(\mu-\mu_0)^n\sigma^n\\
& \lambda^\mu = \lambda^0 + \sum_{n > 0} (\mu - \mu_0)^n \lambda^n
\end{split}
\end{equation}
in such a way that, when we substitute it in \eqref{external force}
and expand (formally)  in powers of $(\mu - \mu_0)^n$ we obtain that the coefficients of same powers match.

Note that this generalizes the standard Lindstedt series, which is a particular case
of the expansion in the case that the family is just $W^\mu = \mu W$ and that we
expand near $\mu_0 = 0$.

If we substitute \eqref{generalexpansion}
in \eqref{external force}, expand in powers of
$\mu - \mu_0$, the coefficient of order $n$ has the form:
\begin{equation}\label{generalordern}
\begin{split}
&v^n_\eta(\psi + \Omega) +
v^n_\eta(\psi - \Omega)
+ v^n_\eta(\psi) (-2 + \partial_\beta W^{\mu_0}((\psi, \eta) ) +\sigma^0v^n+\sigma^nv^0+\lambda^n = R_n
\end{split}
\end{equation}
where $R_n$ is a polynomial expression in $v^0,\ldots, v^{n-1}$.

The main observation is that the equation \eqref{generalordern} is
precisely the Quasi-Newton equations \eqref{eq:I} which can be solved
by factorization. Note that we get the perturbative series to all
orders only assume that $W^{\mu}$ factorizes at $\mu=\mu_0$ and the
expansion series we get has $\sigma=0$ to all orders.

We note that to solve the cohomology equations without any quantitative
estimates, as shown in \cite{SuZL15}, it suffices to assume 
\begin{equation} \label{subexponential}
\lim_{N\to\infty} \frac{1}{N} \sup_{|\tilde{k}| \le N, m \in \mathbb{Z}} \bigg| \ln |\tilde{k}\cdot \Omega-m| \bigg|  = 0,
\end{equation}
which is weaker than Diophantine and, indeed weaker than Brjuno conditions. 

We have, therefore established the following:
\begin{lemma} \label{generalexpansionexists}
Under the assumptions {\bf A0}, {\bf A1}
and that $\Omega$ satisfies the quantitative conditions.

Then, we can find formal power series as in \eqref{generalexpansion}
such that for any
$N \in \mathbb{N}$ and any $\rho'$, $0 < \rho' < \rho$,  we have
\[
\left\| \mathscr{E}_\mu \left[ \sum_{n \le N} v^n (\mu -\mu_0)^n , \sum_{n \le N} \lambda^n (\mu -\mu_0)^n \right]\right\|_{\rho'}
\le C_{N,\rho'} |\mu - \mu_0|^{N+1}.
\]

Furthermore, it is possible to find
a  formal power series
that satisfies the normalization
\[
 \langle v^n_\eta \rangle = 0 \quad \quad n  =1,\ldots N.
\]
We refer to it as the \emph{``normalized perturbative expansion''}.

This normalized perturbative expansion is unique.
\end{lemma}

Similarly, we can obtain existence of perturbation expansions 
for both the equilibrium and the factorization equations. 
In analogy with \eqref{generalexpansion}, we seek expansions 

\begin{equation}\label{generalexpansion2} 
\begin{split}
& v^\mu = v^0 + \sum_{n > 0} (\mu - \mu_0)^n v^n \quad \quad
 \lambda^\mu = \lambda^0 + \sum_{n > 0} (\mu - \mu_0)^n \lambda^n  \\
& c^\mu = c^0 + \sum_{n > 0} (\mu - \mu_0)^n c^n \quad \quad
 \sigma^\mu = \sigma^0 + \sum_{n > 0} (\mu - \mu_0)^n \sigma^n  \\
\end{split}
\end{equation}
in such a way that the equilibrium and factorization equations are solved. 

The equations for order $n$ are 
\begin{equation}\label{ordernboth}
\begin{split} 
&v^n_++v^n_-+(-2+\partial_{\beta}W^{\mu_0}_{v^0}+\sigma^0)v^n+v^0\sigma^n+\lambda^n=R_n \\
&-c_+^0c^n+(-c^0+2-\partial_{\beta}W_{v^0}^{\mu^0}-\sigma^0)c^n_+-c_+^0\sigma^n-\partial_{\beta}\partial_{\beta}W_{v^0}^{\mu_0}c_+^0v^n=\tilde{R_n}.
\end{split}
\end{equation}
From the proof of the KAM theorem, it's clear these can be solved to all orders assuming that $W^{\mu}$ factorizes at $\mu=\mu_0$. The series we get do not have $\sigma=0$, which is different from the series when we do not require the factorization for all orders.

We have, therefore established the following:
\begin{lemma} \label{generalexpansionexists2}
Under the assumptions {\bf A0}, {\bf A1}, {\bf A2}
and that $\Omega$ satisfies the quantitative conditions.

Then, we can find formal power series as in \eqref{generalexpansion2}
such that for any
$N \in \mathbb{N}$ and any $\rho'$, $0 < \rho' < \rho$,  we have
\[
\left\| \mathscr{E}_\mu \left[ \sum_{n \le N} v^n (\mu -\mu_0)^n ,  \sum_{n \le N} \sigma^n (\mu -\mu_0)^n, \sum_{n \le N} \lambda^n (\mu -\mu_0)^n \right]\right\|_{\rho'}
\le C_{N,\rho'} |\mu - \mu_0|^{N+1}
\]
and
\[
\left\| \mathscr{F}_\mu \left[ \sum_{n \le N} v^n (\mu -\mu_0)^n ,  \sum_{n \le N} \sigma^n (\mu -\mu_0)^n, \sum_{n \le N} c^n (\mu -\mu_0)^n \right]\right\|_{\rho'}
\le C_{N,\rho'} |\mu - \mu_0|^{N+1}.
\]

Furthermore, it is possible to find
a  formal power series
that satisfies the normalization
\[
 \langle v^n_\eta \rangle = 0 \quad \quad n  =1,\ldots N.
\]
We refer to it as the \emph{``normalized perturbative expansion''}.

This normalized perturbative expansion is unique.
\end{lemma}

The above results could be interpreted in a more geometric way. We have proved convergence for the perturbative expansion involving also the counterterms that ensure the factorization. Given a potential $W$ small, the main result shows that we can find a locally unique counterterm $\sigma$, which is a functional on $W$ so that the equation factorizes.
The set $\{W(v) + \sigma(W) v\}$ can be interpreted as a codimension 1 manifold (the factorization manifold) in the space of potentials.

The perturbative expansions in Lemma \ref{convergence of the formal expansion} can be interpreted as the perturbative expansions for a path of potentials indexed by $\epsilon$ but requiring that $W$ stays on the factorization manifold.

In this geometric interpretation, the condition \eqref{transversal to factorization manifold} that among $\sigma$ we can achieve the factorization can be interpreted geometrically as saying that the direction (in the space of potentials) given by $\sigma v$ are transversal to the factorization manifold $F$. Note that this factorization manifold may depend on $\omega$.

\subsection{Convergence of Lindstedt series for the equilibrium and factorization equations}
\label{sec:convergenceseries}

\begin{theorem}\label{convergence of the formal expansion}
Assume that the conditions of Lemma~\ref{generalexpansionexists2} hold
and that $\Omega$ is Diophantine. Then, the normalized  formal series
obtained in Lemma~\ref{generalexpansionexists2}
is convergent.
\end{theorem}

Note that a particular case of the above result is the convergence of the Lindstedt series
when the frequency is Diophantine.

We need the frequency to be Diophantine because, as we will see, the proof
uses repeatedly the KAM theorem. It seems quite possible that for the frequencies
that satisfy \eqref{subexponential} but not \eqref{Diophantine}, for many perturbations, it is
possible to obtain perturbative expansions to all orders, which nevertheless do not converge.

\begin{proof}
Without loss of generality, we assume that $v^0$ satisfies \eqref{normalization}.

Using the KAM theorem (Theorem \ref{analytic KAM theorem}),
we obtain that for all $\mu$ in a small ball centered on $\mu_0$ there exists a (unique) normalized
solution. (It suffices to  note that $v^0$ is a sufficiently approximate solution for
all $\mu$ close to $\mu_0$.)

Now, given any $\tilde \mu$ in this small ball, we can obtain a perturbative expansion
in powers of $\mu - \tilde \mu$.

We consider the first term of the expansion $v^1_{\tilde \mu}$.
We remark that it will be uniformly bounded in $|| \cdot||_{\rho'}$.

We note that because
\[
||\mathscr{E}_\mu[ v^0_{\tilde \mu} + (\mu - \tilde \mu) v^1_{\tilde \mu}, \lambda^0_{\tilde \mu}  + (\mu - \tilde \mu) \lambda^1_{\tilde \mu} ]||_{\rho'} \le
C |\mu - \tilde \mu |^2
\]
we can apply the KAM theorem (note that the non-degeneracy conditions of
the KAM theorem are satisfied with uniform bounds when the ball is
considered small enough)  and obtain that there is a normalized  solution $v^*_\mu, \lambda^*_\mu$
of \eqref{external force} for any value of $\mu$ in a ball around $\tilde \mu$ and that it
satisfies
\begin{equation}\label{newsolution}
\begin{split}
&|| v^*_\mu - v^0_{\tilde \mu} + (\mu - \tilde \mu) v^1_{\tilde \mu} ||_{\rho''} \le C|\mu - \tilde\mu|^2\\
&|| \lambda^*_\mu - \lambda^0_{\tilde \mu} + (\mu - \tilde \mu) \lambda^1_{\tilde \mu} ||_{\rho''} \le C|\mu - \tilde\mu|^2.
\end{split}
\end{equation}

Using the uniqueness obtained in the KAM we obtain that $v^*_\mu = v_\mu^0$, $\lambda^*_\mu = \lambda^0_\mu$.
Hence, \eqref{newsolution} means that $v^1_{\tilde \mu} $ is the derivative  at $\mu = \tilde \mu$
of the mapping that
to $\mu$ associates $v^0_\mu, \lambda^0_\mu$ if we give $v$ the topology in $\mathscr{A}_{\rho''}$.

We recall that the Cauchy-Goursat theorem shows that any complex function which is differentiable
at every point, is analytic \cite{Ahlfors}. This argument also works for functions taking values in Banach
spaces.
Alternatively, it  is not difficult to show that the mapping $\mu \mapsto v^1_\mu$ is continuous
(a quick way is to show that the graph of the map is closed because of the uniqueness
and that, since it uniformly bounded, it is compact by Montel's theorem \cite{Ahlfors}).

Once we have that the function is analytic, we know its Taylor series converges, but the Taylor series
has to be the one given by the formal series expansion.
\end{proof}

Notice as a corollary of the dependence on parameters we can obtain that the solution is analytic in the parameter $\eta$.

Since in the physical applications $\eta\in\mathbb{T}^1$ is important to discuss the periodicity in $\eta$ of the solutions thus obtained, we remark that , if we start with an approximate solution which is periodic in $\eta$, we will obtain a solution which is also periodic in $\eta$. 

This can be seen in two ways. One can observe that , applying Theorem \ref{analytic KAM theorem} we can obtain solutions in small enough intervals of $\eta$. They will be analytic in these small intervals and therefore they give a global analytic solution. Furthermore, we observe that if the approximate solution corresponding to $\eta=1$ is close to the solution corresponding to $\eta=0$, they have to agree because of uniqueness. Hence, we obtain that the solution is periodic.

We could also argue that, the proof is based on  an iterative step and that, by examining the proof, all the steps preserve the periodicity in $\eta$ of the approximate solutions. Hence, the KAM procedure that we describe for a fixed $\eta$ lifts to a procedure for periodic functions of $\eta$.

The lifting of the problem to a space of functions of $\eta$ also gives a direct proof of smooth dependence on parameters.

\section*{Acknowledgements}
We thank Professor T. Blass for discussions.
R. L. and L. Z.  have  been supported by DMS-1500943.
The hospitality of
JLU-GT Joint institute for Theoretical Sciences for the three authors
was instrumental in finishing the work. 
 R.L also acknowledges the hospitality of
the Chinese Acad. of Sciences.
X. Su is supported by both National
Natural Science Foundation of China (Grant No. 11301513) and ``the Fundamental Research Funds for the Central Universities".
\bibliographystyle{alpha}
\bibliography{reference}
\end{document}